\documentclass[12pt]{amsart}
\usepackage{color,graphicx,array, amssymb, amscd,slashed}
\usepackage{tikz-cd}
\usepackage{csquotes}
\usepackage{enumitem}
\usepackage{appendix}
\newlist{steps}{enumerate}{1}
\setlist[steps, 1]{label = Step \arabic*:}
\newtheorem{Theorem*}{Theorem}[]
\newtheorem{Theorem}{Theorem}[section]

\newtheorem{Proposition}[Theorem]{Proposition}
\newtheorem{Corollary}[Theorem]{Corollary}
\theoremstyle{definition}

\theoremstyle{remark}
\newtheorem{Example}[Theorem]{Example}
\newtheorem{Remark}[Theorem]{Remark} 

\numberwithin{equation}{section}

\newcommand{\R}{\mathbb R}

\newcommand{\C}{\mathbb C}

\newcommand{\D}{\mathbb D}

\newcommand{\St}{\mathbb S}

\newcommand{\ISU}{{\rm SU}_{1, 1}}
\newcommand{\LISU}{\Lambda {\rm SU}_{1, 1 \sigma}}
\newcommand{\Uone}{{\rm U}_1}
\newcommand{\id}{\operatorname{id}}

\newcommand{\isu}{\mathfrak{su}_{1, 1}}

\newcommand{\Lisu}{\Lambda \mathfrak{su}_{1, 1 \sigma}}
\newcommand{\SL}{{\rm SL}_2 \mathbb C}

\newcommand{\LSL}{\Lambda {\rm SL}_2 \mathbb C_{\sigma}}
\newcommand{\LSLPM}{\Lambda^{\pm} {\rm SL}_{2} \mathbb C_{\sigma}}
\newcommand{\LSLM}{\Lambda^{-} {\rm SL}_{2} \mathbb C_{\sigma}}
\newcommand{\LSLMI}{\Lambda_*^{-} {\rm SL}_{2} \mathbb C_{\sigma}}
\newcommand{\LSLP}{\Lambda^+ {\rm SL}_{2} \mathbb C_{\sigma}}
\newcommand{\LSLPI}{\Lambda_*^+ {\rm SL}_{2} \mathbb C_{\sigma}}

\newcommand{\lsl}{\Lambda \mathfrak{sl}_2 \mathbb C_{\sigma}}
\newcommand{\ad}{\operatorname{Ad}}

\newcommand{\Nil}{{\rm Nil}_3}
\newcommand{\nil}{\mathfrak{nil}_3}

\renewcommand{\Im}{\operatorname {Im}}

\newcommand{\Lt}{\mathbb L_3}
\renewcommand{\l}{\lambda}

\usepackage{amsmath}	
\begin{document}
\title{Minimal cylinders in the three-dimensional Heisenberg group}
\dedicatory{}
%
 \author[S.-P.~Kobayashi]{Shimpei Kobayashi}
 \address{Department of Mathematics, Hokkaido University, 
 Sapporo, 060-0810, Japan}
 \email{shimpei@math.sci.hokudai.ac.jp}
 \thanks{The author is partially supported by Kakenhi 18K03265 and 22K03304.}
 \date{\today}
\pagestyle{plain}
\begin{abstract}
We study minimal cylinders in the three-dimensional 
 Heisenberg group $\Nil$ using the generalized Weierstrass 
 type representation, the so-called loop group method. 
 We characterize all non-vertical minimal cylinders in terms of 
 pairs of two closed plane curves which have the same signed area.
 Moreover, as a byproduct of the construction, 
 spacelike CMC cylinders can also be obtained.
 \end{abstract}
\maketitle    
\section{Introduction and the main result}
 \textbf{(A)} The construction of surfaces with special properties is an important and 
 standard task of classical differential geometry. A particularly important class of surfaces in the 
 three-dimensional Heisenberg group is formed by the minimal surfaces.
 They have been investigated by many authors for many years, e.g.,
 \cite{CPR, Figueroa, FMP, IKOS, Fer-Mira2, Daniel:GaussHeisenbergw}.
 In \cite{DIKAsian} the  basic loop group construction of all such surfaces with contractible domain 
 of definition has been presented. In  \cite{DIK;Nil3-sym} this was generalized to minimal surfaces with non-trivial topology. More generally minimal surfaces with symmetries have been investigated. 
 As an illustration of our technique we have discussed equivariant surfaces and helicoidal surfaces in some detail.
 The present paper continues the discussion of special surfaces in the
 three-dimensional Heisenberg group $\Nil$ by investigating minimal cylinders.

 Before we can describe the present work we would like to recall
 the generalized Weierstrass type representation for non-vertical minimal surfaces in $\Nil$.  For details and notation we refer to \cite{DIKAsian, DIK;Nil3-sym}.
  In this paper we consider exclusively Riemann surfaces $M$ and their universal cover $\widetilde{M} = \D$, a contractible open subset of $\C$, since there is no minimal $S^2$ in $\Nil$.
\begin{steps}
 \item  Choose any holomorphic potential $\eta = A (z, \l) dz$
 so that $A(z, \l) = \sum_{k=-1}^\infty A_k(z) \lambda^k$ is defined for $\l \in \C^*$ and all $z \in \D$ and 
 the diagonal entries are even, and the off-diagonal entries are odd in $\lambda$.
 Moreover, assume that the $(1,2)$-entry of $A_{-1}$ never vanishes.  
 \item  Solve the ordinary differential equation $d C = C \eta$ 
  on $\D$ with initial condition identity at some base point $z_0$.
  Here we consider a holomorphic potential and therefore do not have any problems to solve the ode (Below we extend what is explained here to certain meromorphic $\isu$-potentials).
 
  \item   
   Consider the connected component $\mathcal{I}_{z_0}$, 
   containing $z_0$,  of all points in $\D$ such that
 $C(z, \lambda) $ has for all $\lambda \in S^1$  an  $\ISU$-Iwasawa (twisted) loop group decomposition on 
 $\mathcal{I}_{z_0}$,
 that is, $C = F V_+$, where $F : \mathcal{I}_{z_0}\rightarrow 
  \LISU$ and 
  $V_+ : \mathcal{I}_{z_0} \rightarrow \LSLP$.
 
 Then $F$ is on $\D(z_0)$ the extended frame of some minimal surface 
 in $\Nil$, that is, $F$ takes values in $\LISU$ and 
 its Maurer-Cartan form 
 has the particular $\lambda$-distribution characteristic for the Maurer-Cartan forms of extended frames. 
 Also note, the leading term $V_0$ of $V_+$ is a diagonal matrix and
 we can assume w.l.g. that it has positive diagonal elements. In this case the decomposition $C = F V_+$
 is unique.
 \item  Next  we apply  Sym-type formulas to the extended frame $F$ as follows: First we compute  the spacelike CMC surface in Minkowski space $f_{\Lt}$ and its  
 Gauss map $N_{\Lt}$ respectively by the formulas
 \begin{equation}\label{eq:SymMin}
 f_{\Lt}=-i \l (\partial_{\l} F) F^{-1} 
 - \frac{i}{2} \ad (F) \sigma_3,\quad \mbox{and} \quad
 N_{\Lt}= \frac{i}{2} \ad (F) \sigma_3
 \end{equation}
 with 
 $\sigma_3 = \left( 
 \begin{smallmatrix} 
 1 &0 \\ 
 0 & -1 
 \end{smallmatrix}
 \right)$.
 Then we finally compute  $f^{\lambda}$ by the formula
\begin{equation}\label{eq:symNil}
 f^{\l}:=\Xi_{\mathrm{nil}}\circ \hat{f}\quad\mbox{with}
\quad
 \hat f = 
    (f_{\Lt})^o -\frac{i}{2} \l (\partial_{\l}  f_{\Lt})^d, 
\end{equation}
 where the superscripts ``$o$'' and ``$d$'' denote the off-diagonal and 
 diagonal part, 
 respectively. Moreover,  $\Xi_{\mathrm{nil}}$ denotes the exponential map from $\nil$ to 
 $\Nil$.  Then $f^{\lambda}, {\lambda \in S^1}$ defines a 
 minimal surface in $\Nil$ for each $\lambda \in S^1$.
\end{steps}
 The Sym-type formula  (\ref{eq:symNil}) first was stated in
\cite{Cartier}, and the form of \eqref{eq:symNil} has been given in \cite{DIKAsian}.

 \textbf{(B)} It is clear that the choice of holomorphic potentials $\eta$ is
 the most important task. 
 Since the extended frames of a minimal surface in $\Nil$ 
 and a spacelike CMC surfaces in the Minkowski $3$-space $\Lt$ 
 are the same, see \eqref{eq:SymMin} 
 and \eqref{eq:symNil},
 the construction technique of spacelike CMC surfaces can be 
 applied. Moreover, it has been known that there exists
 a natural choice of holomorphic potentials for CMC-cylinders in
 the Euclidean $3$-space, see \cite{DK:cyl}. Note that the extended frames
 of a CMC surface in $\mathbb E^3$ and a spacelike CMC surface in 
 $\Lt$ just take values in different real forms of the loop group of ${\rm SL}_2 \C$.
 We thus follow the strategy from Section 5.2 in \cite{DK:cyl} for 
 the construction of all non-vertical minimal cylinders in $\Nil$.

 By the appendix we can assume without loss of generality that one can start from a
 holomorphic  and periodic matrix $C_0(z)$ of period $p>0$ of the form 
 \begin{equation}\label{eq:frameperiodic1}
 C_0(z) = \begin{pmatrix} a_0(z) & b_0(z) \\  
 \overline {b_0(\bar z)} & \overline{ a_0(\bar z)} \end{pmatrix}\;\;,
 \end{equation}
 where $a_0(z), b_0(z)$ are periodic functions of period $p >0$ satisfying
 $\det C_0(z) = a_0(z) \overline{a_0(\bar z)} - b_0(z) \overline{b_0(\bar z)} =1$.
 In particular $C_0(z)$ is in $\ISU$ for $z \in \mathbb R$.   Moreover, we assume in addition $C_0(z =0) = \id$.
 \begin{Remark}
 If we also want to admit  the case
  $C_0(z + p) = -C_0(z),$ then we can  change in two ways to the periodic case:
 \begin{enumerate}
  \item[(a)] We change $C_0$ by a gauge 
$k$, where $k$ is a diagonal matrix with entries  $\exp(i \pi z/ p)$
and $\exp(-i\pi z/ p)$. This does not change anything of geometric substance.
However, strictly speaking it changes the original surface to another one by diagonal dressing.

\item[(b)] Or we consider the period $2p$.
 \end{enumerate}
 \end{Remark}
 Set
 \begin{align}\label{eq:C0} 
 \zeta_0 (z) =  C_0^{-1} \partial_z C_0 dz
 =  
 \begin{pmatrix} \nu(z) & \kappa (z) \\ 
 \overline { \kappa (\bar z)} & -\nu(z) 
\end{pmatrix}dz,
\end{align}
 then $\nu$ and $\kappa$ are periodic functions of period $p >0$ 
 and $\overline{\nu (\bar z)} = - \nu (z)$.
 It is easy to see that 
\begin{equation}\label{eq:kappa}
 \kappa (z)  = \overline{a_0(\bar z)}\partial_z b_0 (z)
 - b_0 (z) \partial_z \overline{a_0(\bar z)}  
\end{equation}
 and $\nu (z)  = \overline{a_0(\bar z)}\partial_z a_0 (z)
 - b_0 (z) \partial_z \overline{b_0(\bar z)}$ hold.
 We take a meromorphic function $h$ on $\mathbb D$, which is periodic of period $p$, 
 and set 
\begin{equation}\label{eq:frameperiodicPot}
 \zeta (z, \lambda) = 
  \begin{pmatrix} 
   \nu (z) & \lambda^{-1}h(z) +  \lambda(\kappa(z) - h(z)) \\ 
 \lambda^{-1}(\overline{\kappa (\bar z)}- \overline{h(\bar z)}  ) 
 + \lambda \overline{ h(\bar z)} & -\nu (z)
   \end{pmatrix} dz.
 \end{equation}
 Then $\zeta (z, \lambda =1) = \zeta_0 (z)$ is meromorphic for 
 $z \in \D$ and in $\isu$  for $z \in \R \cap \mathbb D$, that is, 
 $\zeta_0$ is $\isu$ on $\R \cap \mathbb D$.

 The potential $\zeta$ satisfies the 
 assumption stated in Proposition 5.1 of \cite{DK:cyl} after  replacing  ``skew-hermitian''
 by \textquote{$\isu$}, and will be called an \textit{$\isu$-potential}
 .
 
 Using the $\isu$-potential $\zeta$ as the input data for 
 Step $1$ above, one can construct a minimal surface in $\Nil$ with 
 particularly nice properties. We will explain them in Section \ref{sc:cylinder} below. 
 
\begin{Remark}
 The $\isu$-potential 
  $\zeta$ is meromorphic and its poles can be off from the real line. In general poles of $\zeta$ 
 give special properties of the resulting minimal surface, for examples, 
 ends, branch points or smooth points etc.
 
\begin{enumerate}
 \item[(a)] Branch points: At a point $z_1$ where the product of the two (off-diagonal) 
 terms of $A_{-1}$ vanishes, considering the partial derivatives of (the Sym formula for)
 the minimal surface in  $\Nil$ immersion derived from the potential. 
 It is easy to see that the differential of the \textquote{immersion} is $1$ or $0$. 
 Thus we obtain a branch point.
 
\item[(b)] Ends: From \cite{DIK;Nil3-sym} one infers that sometimes poles of specific structure 
 characterize some also characterize \textquote{ends}.
\end{enumerate}
 \end{Remark}
 
\textbf{(C)} To state the main result of this paper, we introduce 
 two complex functions $\ell$ and $m$ defined by 
\begin{align}\label{eq:ell1}
\ell (t) &=\int_{0}^{t} \left\{a_0 (s)^2 (2 h(s) - \kappa (s) )
 + b_0 (s)^2 (2 \overline{h(s)} - \overline{\kappa (s)}) \right\}ds \\
   m(t)&= a_0(t) b_0(t), \label{eq:m}
\end{align}
  for $t \in \R \cap \mathbb D$.
 Note that $\ell$ and $m$ give plane curves in $\C$, and 
 the curve $m$ is determined by the periodic matrix $C_0$ 
 only, and in particular, the curve $m(t)$ is automatically periodic and thus closed.
 
 Now we can formulate the main result of this paper.
\begin{Theorem*}\label{thm:main}
 The following statements hold$:$
 
 {\rm (i)} Let $C_0$  and  $\zeta$ be a periodic matrix and (after choosing a function $h$) the corresponding $\isu$-potential as defined in  \eqref{eq:frameperiodic1} and \eqref{eq:frameperiodicPot}, 
  respectively, 
  and let $\ell$ and $m$ be the plane curves defined in \eqref{eq:ell1} and \eqref{eq:m} respectively.
  Then $\zeta$ gives, via the generalized Weierstrass representation explained in part {\rm (A)} above, 
   a non-vertical minimal cylinder in $\Nil$ with possibly singular points
  if the curves $\ell$ and $m$ are closed and their
  signed enclosed areas are equal.
 
  {\rm (ii)} Let $\tilde \ell$ and $\tilde m$ be analytic closed plane curves 
   whose signed enclosed areas are equal. Then there exist a periodic matrix 
    $C_0$ and an $\isu$-potential $\zeta$ such that the plane curves 
  $\ell$ and $m$ in \eqref{eq:ell1} and \eqref{eq:m}
  are $\tilde \ell$ and $\tilde m$  respectively, and $\zeta$ gives a non-vertical 
  minimal cylinder in $\Nil$ via the generalized Weierstrass representation explained in part {\rm (A)} above.
  
  {\rm (iii)} All non-vertical minimal cylinders in $\Nil$ can be constructed from
   two analytic curves which satisfy the conditions in {\rm (ii)}.
\end{Theorem*}
 The proof will be given in Section \ref{subsc:proof} below.

 \begin{Remark}
 One of the closed curves $\ell$ or $m$ can be degenerate, i.e. for example in Figure \ref{fig:0}, $m=0$ while $\ell$ 
 is 
 a regular curve. On the one hand, in Example \ref{ex:cyl4}
 \end{Remark}

\begin{figure}[htbp]
  \begin{center}
    \begin{tabular}{c}

      \begin{minipage}{0.45\hsize}
        \begin{center}
   \includegraphics[width=0.6\textwidth]{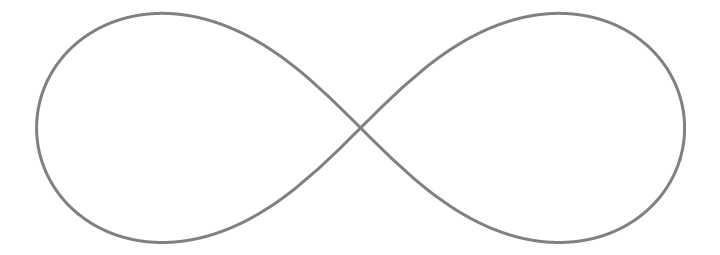}
        \end{center}
      \end{minipage}

      \begin{minipage}{0.55\hsize}
        \begin{center}
   \includegraphics[width=0.7\textwidth]{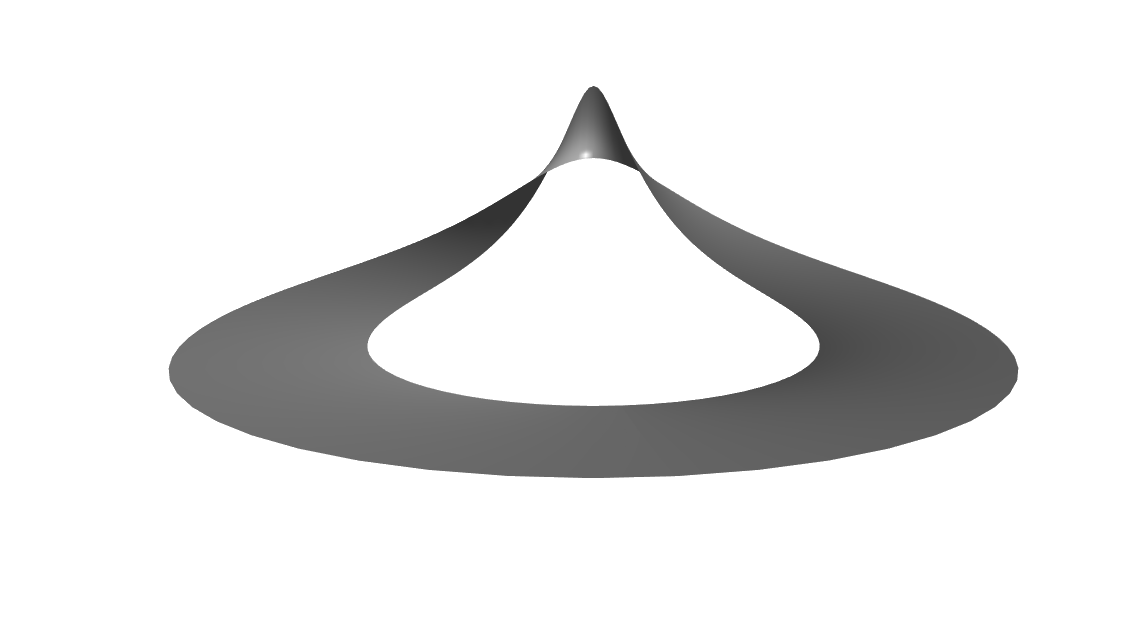}
        \end{center} 
      \end{minipage}

    \end{tabular}
   \includegraphics[width=0.3\textwidth]{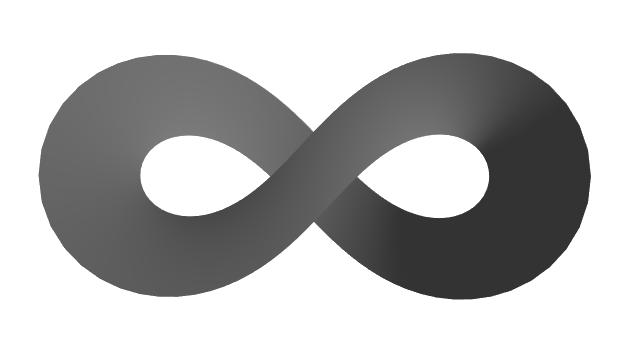}
    \caption{The closed plane curve $\ell$ (left, the lemniscate) and two views of the corresponding 
 non-vertical minimal cylinder in $\Nil$ (right, bottom) from 
  the 
  potential $\zeta$ stated in \eqref{eq:zetapotential} with 
 $h$ defined by first one in \eqref{eq:choiceh0and1} and  initial frame $C_0(z) \equiv \id$. 
 Note, we have $T = 0$ in this case, and the curve $m$ is degenerate to a point $0$.
  The figures were drawn by using \cite{Br:Matlab} developed by David Brander.}
    \label{fig:0}
  \end{center}
\end{figure}

\section{Minimal cylinders in $\Nil$}\label{sc:cylinder}
\subsection{Constructing $C_0$ and $\zeta$ for non-vertical minimal surfaces in $\Nil$}
We show first the last part of the theorem, i.e. that each non-vertical minimal cylinder in $\Nil$ can be obtained from a potential $\zeta$ defined by using
a matrix $C_0$ and a function $h$ as discussed in the introduction.

 We follow the scheme outlined in  \cite{DIK;Nil3-sym} and change \textquote{Mutatis Mutandis}.
 First we recall that we  have shown in Section 2.3.1 loc.cit. that 
 after multiplication by some diagonal matrix $L$ the 
 extended frame $F$ has a meromorphic extension $U = FL$.
 In Section 2.3.2 loc.cit. we have discussed the extension of symmetries 
 to the meromorphic extension. 
 The closing condition is stated in Theorem 2.11. loc.cit.
 We thus have proven the results of Section 3 of  \cite{DK:cyl}.
%

\subsection{The  construction of frame periodic minimal immersions in $\Nil$ from $C_0$ and $\zeta$ and $h$} We have already discussed the construction of a minimal surface in $\Nil$ from the initial  data $C_0$ and $h$ in (B) and (C) of the introduction.

 Recall, if  $\zeta$ is the potential in \eqref{eq:frameperiodicPot} and 
 let $C = C(z, \lambda)$  denote the solution to $d C = C \zeta$, 
 $C(0, \lambda) =\id$. 
 Then it is easy to see that $C$ takes values in $\LISU$ for $z \in \R$, 
 since $\zeta$ takes values in $\Lisu$ for $z \in \R$.
 Our construction implies $C(z, 1) = C_0 (z)$, whence the first closing condition in Theorem 2.11 of \cite{DIK;Nil3-sym}
 is satisfied, that is, 
\begin{equation}\label{eq:firstclosing}
 M(\lambda)\in \LISU\quad\mbox{and}\quad M(\lambda = 1) = \pm \id
\end{equation} 
 for the monodromy matrix $M$ defined by $C(z+p, \l) = M(\l) C(z, \l)$.
 Moreover, $C(z, \lambda)$ is the extended frame along $z \in \R$ up to 
 a $\Uone$ factor, that is $F= C k$ for some $k \in \Uone$ on $z \in \R$.
 Then the extended frame $F$ of a minimal  surface in $\Nil$
 satisfies $F(z+ p, \lambda) = M(\lambda) F(z, \lambda) k(z)$ and \eqref{eq:firstclosing}.
 Such a surface will be called a \textit{frame 
 periodic minimal immersion in $\Nil$.}
 
 The  following theorem  is an exact analogue of 
  Theorem 5.4 in \cite{DK:cyl} for $\isu$-potentials
  and the proof is as well.
 \begin{Theorem}\label{thm:framePeriodic}
 Let $C_0$ be given by \eqref{eq:frameperiodic1}.
 Set $\zeta_0 = C_0^{-1} d C_0$ and define a $\isu$-potential $\zeta$ by 
 \eqref{eq:frameperiodicPot}. 
 Then $\zeta$ defines a frame 
 periodic minimal immersion into $\Nil$.
 Moreover, every frame periodic minimal immersion in $\Nil$ can be obtained this way.
 \end{Theorem}
 
 \subsection{The remaining two closing conditions for minimal cylinders in $\Nil$}\label{subsc:secondthird}
 From Theorem 2.11 of \cite{DIK;Nil3-sym} we derive that in addition to the condition of 
 frame periodic we also need to satisfy two more closing conditions
 to obtain a minimal \textquote{cylinder} in $\Nil$:
 \begin{equation}\label{eq:closing}
 X^o(\lambda =1) = O \quad \mbox{and} \quad Y^d(\lambda =1)= O,
\end{equation}
 where $ X = -i \l (\partial_{\l} M) M^{-1}$ and 
 $Y =  - \frac{1}{2}  \lambda \partial_{\lambda} (\lambda (\partial_{\lambda} M) M^{-1})$,
 considered as functions of $\lambda$, and where the superscript ``$d$'' and ``$o$''  means ``diagonal part'' 
 and ``off-diagonal part'' respectively.
 It is natural to call $M(\l)= \pm \id$ the \textit{first} closing condition.
 Accordingly, the condition $X^o(\lambda =1) = O $ and 
 the condition $Y^d(\lambda =1) = O$ will be called the \textit{second} and 
 the \textit{third} closing condition, respectively.
 
 To consider the second and third closing conditions, we next 
 prove \cite[Theorem 5.1.2]{Kil} in our setting.
 
 \begin{Proposition} \label{Kiliansformula}
 Let $C$ be the solution of $d C = C \zeta$ with $C(0, \lambda) =\id$ 
 where the $1$-form
 $\zeta$ is given as before by \eqref{eq:frameperiodicPot}. Then the following formula holds$:$
 \[
  (\partial_{\l} C(z)) C(z)^{-1} = \int_0^z C(t, \lambda) (\partial_{\l}  \zeta (t,\lambda) ) C(t, \lambda)^{-1}
 \]
 for all real $z$ and integration along the real axis.
\end{Proposition}
\begin{proof}
 Consider the $z$-derivative of $(\partial_{\l} C) C^{-1}$: 
\begin{align*}
 \partial_{z} ((\partial_{\l} C) C^{-1})  dz&= ( \partial_{\l}( \partial_{z} C) ) C^{-1}dz -
 (\partial_{\lambda} C) C^{-1} (\partial_{z}C) C^{-1}dz \\& =  ( \partial_{\l}( C \zeta) ) C^{-1} -
 (\partial_{\lambda} C) C^{-1} (C \zeta) C^{-1}\\&=  C (\partial_{\lambda} \zeta) C^{-1}.
\end{align*}
 Now an integration from $0$ to $z$ along the real axis yields 
\[
 (\partial_{\l} C(z, \lambda)) C(z, \lambda)^{-1} = \int_0^z C(t, \lambda) (\partial_{\l}  \zeta (t,\lambda) ) C(t, \lambda)^{-1} + A(\lambda).
\] 
 But since $C(0,\lambda) = \id$, substituting $z = 0$ above yields 
 $A(\lambda) = 0$ for all $\lambda$. This completes the proof.
\end{proof}
Since $C(z+p, \lambda) = M(\lambda) C(z,\lambda)$ implies $M(\lambda) = C(p, \lambda)$,
we infer
\begin{Corollary} \label{closingcondition}
 Let 
 $X(\lambda) = -i \l (\partial_{\l} M(\lambda )) M(\lambda)^{-1}$ and 
 $Y(\lambda) =  - \frac{1}{2}  \lambda \partial_{\lambda} (\lambda (\partial_{\lambda} M(\lambda ))
  M(\lambda )^{-1})$ respectively. Then the 
 following formulas hold$:$
\begin{align} 
 X(\lambda) &=  -i  \int_0^p C(t, \lambda) (\l \partial_{\l}  \zeta (t,\lambda) ) C(t, \lambda)^{-1},  \\
 Y(\lambda ) &=- \frac{i}{2} \l \partial_{\l} X(\lambda) 
 = - \frac12 \lambda \partial_{\lambda}
\left(\int_0^p C(t, \lambda) (\l \partial_{\l}  \zeta (t,\lambda) ) C(t, \lambda)^{-1}\right).
\end{align}
\end{Corollary}
 Using Proposition \ref{Kiliansformula} and Corollary \ref{closingcondition}, 
 it is straightforward to rewrite the second and third closing conditions for $\lambda = 1$ in a simple way.
\begin{Proposition}
 Let $\zeta$ be as before a $\isu$-potential as 
 given in  \eqref{eq:frameperiodicPot}.
 Moreover, let $\alpha$ be the integrand in \eqref{eq:ell1},
 \begin{equation}\label{eq:alpha}
 \alpha(s)= a_0 (s)^2 (2 h(s) - \kappa (s) )
 + b_0 (s)^2 (2 \overline{h(s)} - \overline{\kappa (s)}).
 \end{equation}
 Then the second closing condition $X^o = 0$ is for 
 $\lambda = 1$ equivalent to
 \begin{equation}\label{eq:secondclosing}
\int_0^{p}  
 \alpha(t) dt  =0. 
 \end{equation}
 Moreover, the third closing condition $Y^d = 0$ is for $\lambda = 1$
 equivalent to
\begin{equation}\label{eq:thirdclosing}
  \int_0^p \Im \left(
 \alpha(t) \left\{\int_0^t \overline{\alpha(s)} ds
\right\}\right)  dt = \int_0^p \Im \left(a_0(t) \overline{b_0(t)} \kappa (t)\right) dt.
\end{equation}
\end{Proposition}
\begin{proof}
 Equation \eqref{eq:secondclosing} follows from 
 Corollary \ref{closingcondition},
 since the integrand of $X(\lambda = 1)$ is
 \[
 C_0(t) \left\{\l \partial_{\l} \zeta (t,\lambda)
 \right\}|_{\lambda = 1} C_0(t)^{-1}
  \] 
  with $(1,2)$-entry equal to $-i\alpha(t)$.
 To evaluate the third (and last) closing condition we need to 
 differentiate first of all $X(\lambda)$ for $\lambda$, 
 and then evaluate at $\lambda = 1$.
\begin{align*}
2 Y(\lambda) &= -i \lambda \partial{_\lambda}X(\lambda) 
 \\ &=  - \int_0^p \left\{\lambda \partial{_\lambda}C(t, \lambda)\right\}
  \left\{\l \partial_{\l}  \zeta (t,\lambda)\right\} C(t, \lambda)^{-1} \\
 & \quad 
 -  \int_0^p C(t, \lambda) [\lambda \partial{_\lambda} \left\{\l \partial_{\l}  \zeta (t,\lambda) \right\} ] C(t, \lambda)^{-1}\\
& \quad +
  \int_0^p C(t, \lambda) \left\{\l \partial_{\l}  \zeta (t,\lambda) \right\} C(t, \lambda)^{-1}
 \left\{\lambda \partial{_\lambda}C(t,\lambda)\right\} C(t, \lambda)^{-1}
\end{align*}
 A straightforward computation shows  
\[
 \lambda \partial{_\lambda} \left(\l \partial_{\l}  \zeta (t,\lambda) \right)
|_{\l =1}
 = \begin{pmatrix} 0 &\kappa(t) \\ \overline{\kappa(t)} & 0 \end{pmatrix} dt.
\]
 Recall that $C(t, \lambda = 1) = C_0(t) = \left( \begin{smallmatrix}
 a_0(t) & b_0(t) \\  \overline{ b_0(t)} & \overline{a_0(t)}   
						  \end{smallmatrix}\right)$ holds. 
 Therefore  the $(1, 1)$-entry of the second summand can be computed at $\lambda = 1$ as
\begin{equation}\label{eq:secondsummand}
 \int_0^p
\left(
a_0(t) \overline{b_0(t)} \kappa (t)
-\overline{a_0 (t)} b_0(t) \overline{\kappa (t)}
\right) dt.
\end{equation} 

 Moreover, the first and third summands can be
 computed at $\lambda = 1$ by using Proposition 
 \ref{Kiliansformula} as
\begin{align*}
& - \int_0^p \left\{\int_0^t C_0(s) (\partial_{\l}  \zeta (s,\lambda) ){|_{\lambda = 1}} C_0(s)^{-1} \right\} 
 \left\{C_0(t) (\l \partial_{\l}  \zeta (t,\lambda) ){|_{\lambda = 1}} C_0(t)^{-1}\right\} \\ 
&+
 \int_0^p \left\{C_0(t) (\l \partial_{\l}  \zeta (t,\lambda) ){|_{\lambda = 1}} C_0(t)^{-1} \right\}
  \left\{ \int_0^t C_0(s) (\partial_{\l}  \zeta (s,\lambda) ){|_{\lambda = 1}} C_0(s)^{-1} \right\}.
\end{align*}
 Putting $A(t) = C_0(t) (\partial_{\l}  \zeta (t,\lambda) ){|_{\lambda = 1}} C_0(t)^{-1}$
 this is equivalent to 
\begin{equation}\label{eq:thridclosing2}
 \int_0^p \left[A(t), \int_0^t A(s)\right].
\end{equation}
 where $[P, Q] = P Q -Q P$ denotes the matrix commutator.
 It is easy to compute the $(1,1)$-entry of \eqref{eq:thridclosing2} as
\begin{align}
\label{eq:thirdsummand}
- \int_0^p 
\left(\alpha (t)
 \int_0^t\overline{\alpha(s)} ds
 - \overline{\alpha(t)}
 \int_0^t\alpha (s)ds \right)dt.
\end{align}
 Putting together \eqref{eq:secondsummand} and 
 \eqref{eq:thirdsummand}, the condition  \eqref{eq:thirdclosing}
 for $Y(\lambda = 1)^d = 0$ follows. This completes the proof.
\end{proof}
\begin{Remark}
 In general the closing conditions in \eqref{eq:closing} can be 
 relaxed as follows:
\begin{equation}\label{eq:nclosing}
 \tilde X^o (\lambda =1)=0 \quad \mbox{and} \quad 
 \tilde Y^d (\lambda =1)=0,
\end{equation}
 where $\tilde X =- i \lambda (\partial_{\lambda} \tilde M) \tilde M^{-1}$
 and $\tilde X =-\tfrac12 \partial_\lambda (\lambda (\partial_{\lambda} \tilde M) \tilde 
 M^{-1})$ with $\tilde M = \left\{M(\lambda)\right\}^n$
 for some $n \in \mathbb N$. Then the second and the third closing condition in 
 \eqref{eq:secondclosing} and  \eqref{eq:thirdclosing} respectively  can be 
 stated by replacing the period \textquote{$p$}
 by \textquote{$np$}, for some
  $n \in \mathbb N$.
\end{Remark}

\subsection{The proof of Theorem \ref{thm:main}}\label{subsc:proof}
 \begin{proof}
 We assume that the two curves $l$ and $m$ are closed and have the same signed area.
  We want to show that the closing conditions for a minimal cylinder in $\Nil$ are satisfied.
  Recall that the first closing condition is satisfied by construction.
  
  (i) We  look at the closing conditions \eqref{eq:secondclosing} and  
 \eqref{eq:thirdclosing} more closely.
 It is evident that the second closing condition, i.e., \eqref{eq:secondclosing} 
 is equivalent to 
 that $\ell$ is a closed curve with  period $p>0$. 
 It thus remains to consider the third closing condition.
 But by using 
 the relation (\ref{eq:kappa}) between  $\kappa$, $a_0$ and $b_0$  and 
 integration by parts,  the third closing condition, 
\eqref{eq:thirdclosing},
 is easily seen to be equivalent to 
\[
\int_0^{p}\Im  \left( \overline{\ell(t)} \ell^{\prime}(t) \right) dt  =
\int_0^{p}\Im  \left( \overline{m(t)} m^{\prime}(t) \right) dt.
\]
 On the left-hand side let $\ell(t) = x(t) + i y (t)$ be 
 written in terms of its real 
 and its imaginary part.
 Then a straightforward computation shows that 
\begin{align*}
 \int_0^{p} \overline{\ell(t)} \ell^{\prime}(t)  dt&= 
\frac12 \int_0^{p} (x(t)^2 + y(t)^2)^{\prime} dt -i \int_0^{p} (x(t)y(t))^{\prime} dt 
+ 2 i \int_0^{p}x(t)y^{\prime}(t) dt 
\end{align*}
showing that the original integral expresses (up to the factor $2i$) the signed area of $\ell$.
 The analogous argument applies to the curve $m$. Since we have assumed that the signed areas of $\ell$ and $m$ are equal, the two integrals above are equal and thus the third closing condition is satisfied.

 (ii) First we define  periodic functions $a_0$ and $b_0$
 by the  equations:
 \begin{equation}\label{eq:a0b0}
  \tilde m(z) = a_0(z) b_0(z), \quad a_0(z) \overline{a_0(\bar z)} - b_0(z) \overline{b_0(\bar z)} =1.
 \end{equation}
  Note, we first solve for functions defined on $\R$ and extend holomorpically, since our original functions were assumed to be real-analytic.
  Note that the functions $a_0$ and $b_0$ are determined up to some phase factors, i.e.,
  $a_0 e^{i \theta}$ and $b_0 e^{-i \theta}$ for some real function $\theta$.
  Then the periodic matrix $C_0$ is defined by 
  \[
   C_0(z) =\begin{pmatrix}
   a_0(z) & b_0(z) \\ 
   \overline{b_0(\bar z)} & \overline{a_0(\bar z)} 
   \end{pmatrix} k, \quad k =
   \begin{pmatrix}
   e^{i \theta(z)} & 0 \\
   0 & e^{-i \theta(z)}
   \end{pmatrix}.
  \]
 The diagonal matrix $k$ is just a gauge 
 and thus the functions $a_0$ and $b_0$
 are uniquely determined by \eqref{eq:a0b0}. 
  Actually, different $C_0$ may yield the same cylinder: If we have $\tilde C_0 = C_0 k,$ with
 $k$ diagonal unitary, then the Iwasawa decomposition yields a new $F$, namely $Fk$. But the Sym-formula
 ignores $k$.
 
 Next the plane curve $\tilde \ell$ is real analytic 
 and also $\tilde \ell^{\prime}(t)$ is.
 Thus $\tilde \ell^{\prime}(t)$  can be extended holomorphically 
 around $t \in \R$ and we denote this curve by $\alpha$.
 Then we define a function $\mu$  
  by 
\begin{equation}\label{eq:mu}
 \mu (z) = \frac1{a_0 (z) \overline{a_0(\bar z)} + b_0 (z) \overline{b_0(\bar z)}} 
  \left(\overline{a_0(\bar z)}^2 \alpha(z) - b_0(z)^2 \overline{\alpha(\bar z)} 
  \right),
\end{equation}
 where $a_0, b_0$ are entries of $C_0$. From the construction it is easy to see that
\[
 \alpha(z) = a_0(z)^2  \mu(z) + b_0(z)^2 \overline{\mu(\bar z)}
\]
 holds. 
 Then we define $h$ to be
\[
 h(z) = \frac{\kappa(z) + \mu(z)}{2},
\]
 where the function $\kappa(z)= \overline{a_0(\bar z)}\partial_z b_0 (z)
 - \partial_z \overline{a_0(\bar z)} b_0 (z)$ is given by 
  \eqref{eq:kappa}, and define the potential $\zeta$ in \eqref{eq:frameperiodicPot} by the data $C_0$ and $h$.
 Then the plane curve $\tilde \ell$ is exactly $\ell$ in \eqref{eq:ell1}
  and $\zeta$ clearly produces a non-vertical minimal cylinder in $\Nil$.

 \textrm{(iii)} To prove the final statement, we consider a non-vertical minimal cylinder in $\Nil$.
 Then without loss of generality, there exists the extended frame $F \in \LISU$ 
 such that  $F$ is periodic on $\R$ 
 with period $p>0$ and $F(z+p, \lambda) = M(\lambda) F(z)$ 
 and the monodromy $M(\lambda)$ satisfies all closing conditions, that is, $M|_{\lambda=1} 
 =\pm \id$,
 $X^o|_{\lambda=1 }=O$ and $Y^d|_{\lambda=1}=O$. We have shown in Section 2.3.1
  in \cite{DIK;Nil3-sym} that  there exists some diagonal 
 matrix $L$ such that $U= FL$ has a meromorphic extension, see e.g., Theorem 3.2 in \cite{DK:cyl}.  Then the Maurer-Cartan 
 form of $U$ is a  potential $\zeta = U^{-1} d U$
 of the desired type, and the two plane curves $\ell$ and $m$ 
 given by $C_0$ and $\zeta$ satisfy the properties in (ii).
 \end{proof}

\subsection{Examples of non-vertical minimal cylinders in $\Nil$ constructed from matrices $C_0$}\label{subsc:diagonal}
 In this subsection, we illustrate the theory presented in the previous section by constructing
 non-vertical minimal cylinders  in terms of simple specific periodic matrices $C_0$ and specific functions $h$.
 
 Thus we start by considering a diagonal matrix 
\begin{equation}\label{eq:diagC0}
C_0(z) =
  \begin{pmatrix} 
  e^{icz} & 0 \\ 
  0 & e^{-icz}
   \end{pmatrix},
\quad c \in \R.
\end{equation}
 Note that $a_0(z) = e^{icz}$ and $b_0(z) =0$ in $C_0(z)$.
 If $c \neq 0$, then we have the minimum period $p$ of $C_0(z)$ given by 
 $\pi/ |c|>0$, and if $c=0$, then we have any period.
We compute 
\begin{align} 
 \zeta_0 (z)& =  C_0 (z)^{-1} \partial_z C_0(z) dz= 
 \begin{pmatrix}
  ic & 0 \\ 
  0 & -ic 
  \end{pmatrix} dz.
\end{align}
 We note that in our case we have $\kappa(z) \equiv 0$ (the off-diagonal part of $\zeta_0$ is 
 identically zero). 
 For the construction of $\zeta$ we still need to choose a periodic function $h(z)$ of period $p$, which must be a positive integer 
  multiple of the minimum period of $C_0$. Then we obtain
 \begin{equation}\label{eq:zetapotential}
 \zeta (z, \lambda) = 
  \begin{pmatrix} 
  i c &   \lambda^{-1} h(z) - \lambda h(z) \\
   -\lambda^{-1} \overline{h(\bar z)}+ \lambda \overline{h(\bar z)} & -ic 
   \end{pmatrix} dz.
 \end{equation}
 From Corollary \ref{closingcondition} we derive that   $X(\lambda = 1)$ is off-diagonal and 
 the integrals of both entries are complex conjugate.
 Since the first closing condition is always satisfied by our construction we only need to consider the second and  the third closing condition. 
 For this we compute the functions $\alpha$ and $m$ in \eqref{eq:alpha} 
 and \eqref{eq:m}.
 We obtain
\begin{equation}\label{eq:alphabeta}
 \alpha(t) =  2 e^{2 i c t} h(t) \quad \mbox{and} \quad 
 m(t) = 0. 
\end{equation}
\begin{Example}[The case $c = 0$]\label{ex:cyl1}
We construct  non-vertical minimal cylinders from  the constant frame $C_0(z) \equiv \id$.
 For this we only need to choose 
 a periodic holomorphic function $h$. 
 
 We present two examples, by choosing
\begin{align}
\label{eq:choiceh0and1}
 h(z) = \frac{1+ i \sin z}{(i + \sin z)^2}, \quad \mbox{or}\quad
 h(z) = \cos(z) - i \sin (3 z).
\end{align}
 Note that we choose the period $p$ 
 for both choices of $h$ as $p = 2 \pi$. 
 For our first choice of $h$
 the curve $\ell(t)= \int_0^t 2h(s) ds$ 
 is the  lemniscate, that is, 
\[
 \ell(t) = - \frac{2 \cos z}{i + \sin z},
\]
 and it is clear 
  that the signed area of $\ell$ is zero.
 Moreover, the resulting non-vertical 
 minimal cylinder does not have any branch point on $\R$, since $h^2(z)+$\textquote{positive function} is the conformal factor of the metric 
  and $h$ in \eqref{eq:choiceh0and1} does vanish on $\R$.
  This example is shown in Figure \ref{fig:0}.
 
 For our second choice of $h$ the curve $\ell (t)$ can be computed as
 \[
  \ell(t) = \frac{2i}3 \cos 3 t + 2 \sin t, 
 \]
 and the signed area of $\ell$ is zero.
 The plane curve $\ell$ and the minimal cylinder given by our second choice of $h$ as stated in \eqref{eq:choiceh0and1}
 are shown in Figure \ref{fig:1}.
\end{Example}

\begin{figure}[htbp]
  \begin{center}
    \begin{tabular}{c}

      \begin{minipage}{0.45\hsize}
        \begin{center}
   \includegraphics[width=0.6\textwidth]{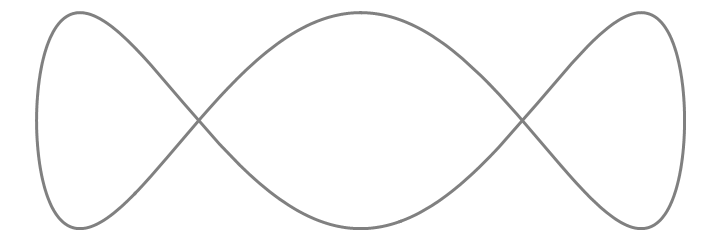}
        \end{center}
      \end{minipage}

      \begin{minipage}{0.55\hsize}
        \begin{center}
   \includegraphics[width=0.8\textwidth]{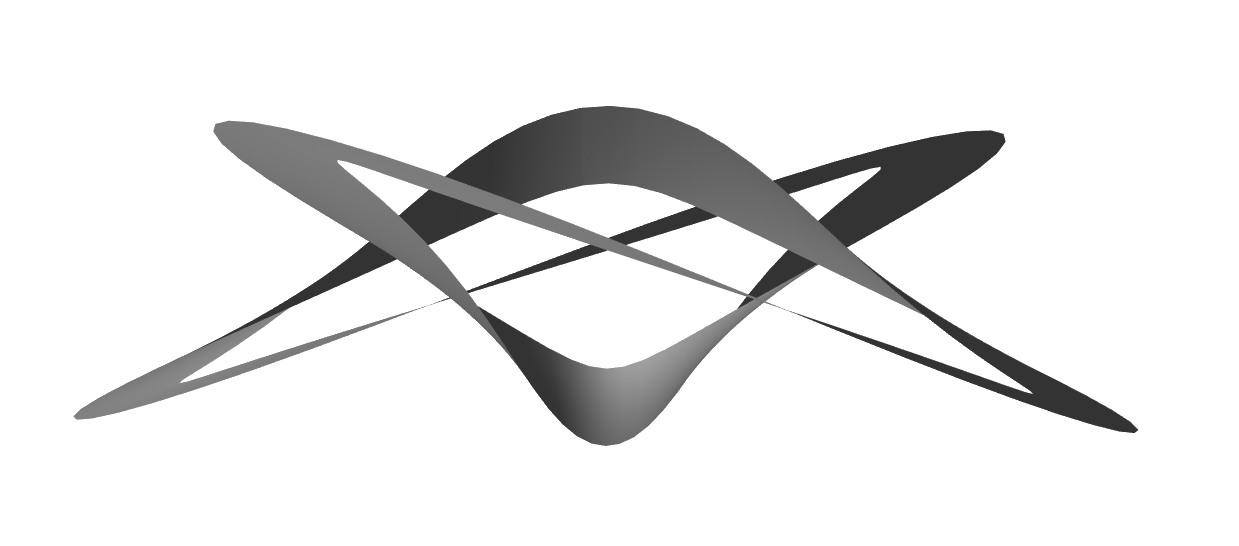}
        \end{center}
      \end{minipage}

    \end{tabular}
    \caption{A closed plane curve $\ell$ (left) and the corresponding 
 non-vertical minimal cylinder in $\Nil$ (right) from 
  the $\isu$-potential $\zeta$ in \eqref{eq:zetapotential} with 
 the second choice of $h$ in \eqref{eq:choiceh0and1} and identity initial condition. The signed 
 area enclosed by $\ell$ is zero in this case.
 The figure is constructed by \cite{Br:Matlab}.}
    \label{fig:1}
  \end{center}
\end{figure}

\begin{Example}[The case $c = 1$] \label{ex:cyl2}

We construct a on-vertical minimal cylinder in $\Nil$  from  the simple periodic frame $C_0(z)$ as stated in (\ref{eq:diagC0}) for $c = 1$.
 We choose 
 the periodic holomorphic function $h$ 
\begin{equation}\label{eq:choiceh2}
 h(z) = \exp(-i \pi/4)+\sqrt{6} \cos (4z).
\end{equation}
 Note that the minimal period is $\frac{\pi}{2}$, but 
 we choose the period $p = \pi$,
 and the plane curve $\ell$ is given by 
\[
\ell(t) = 
\frac16 e^{-2 i t}\left(3 \sqrt 6 i  - \sqrt 6 i e^{8 i t}
 - 6 \sqrt i e^{4 i t} \right).
\]
 It is easy to see that $\ell$ is closed curve whose signed area is
  zero for the period $p$. 
 Therefore the resulting minimal surface in $\Nil$ is a cylinder. 
 Moreover, the resulting non-vertical 
 minimal cylinder does not have any branch point on $\R$, since $h^2(z)+$\textquote{positive function} is the conformal factor of the metric 
  and $h$ in \eqref{eq:choiceh2} does vanish on $\R$.
 The corresponding curve $\ell$ and the corresponding surface are shown in Figure \ref{fig:2}

\begin{figure}[htbp]
  \begin{center}
    \begin{tabular}{c}

      \begin{minipage}{0.45\hsize}
        \begin{center}
   \includegraphics[width=0.6\textwidth]{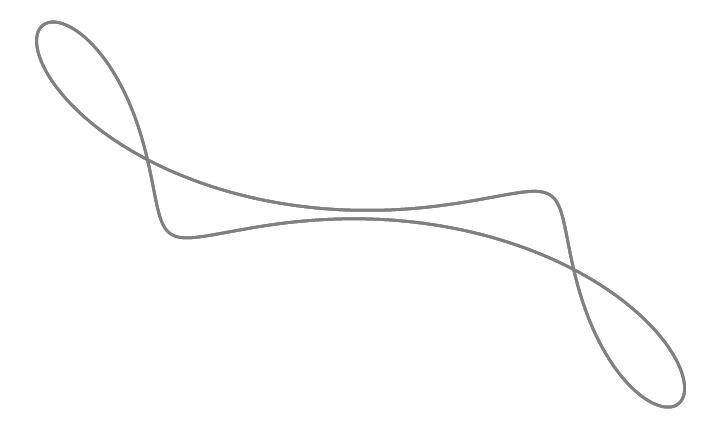}
        \end{center}
      \end{minipage}

      \begin{minipage}{0.55\hsize}
        \begin{center}
   \includegraphics[width=0.75\textwidth]{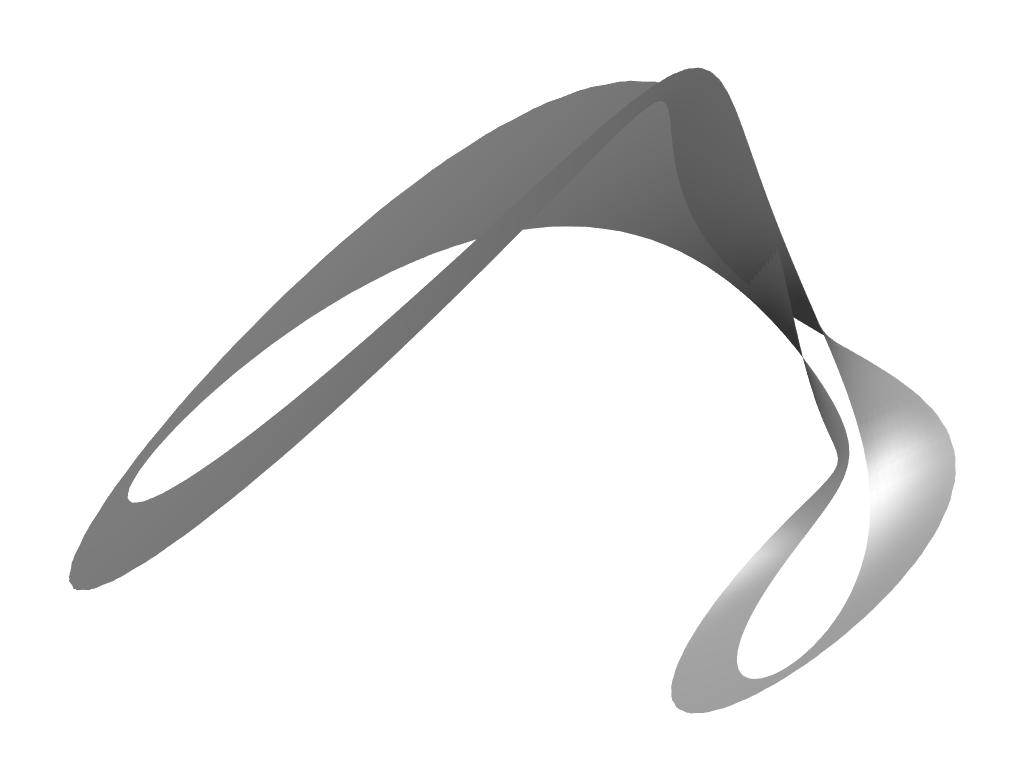}
        \end{center}
      \end{minipage}

    \end{tabular}
    \caption{The closed plane curve $\ell$ (left) and the corresponding 
  non-vertical minimal cylinder in $\Nil$ (right) constructed from 
  the $\isu$-potential $\zeta$ stated in \eqref{eq:zetapotential} with 
 $h$ as defined in \eqref{eq:choiceh2} and
 $C_0(z)$ defined as in (\ref{eq:diagC0}) for $c = 1$. The signed 
 area enclosed by $\ell$ is zero in this case.
 The right figure is constructed by \cite{Br:Matlab}.}
    \label{fig:2}
  \end{center}
\end{figure}
\end{Example}


\subsection{An example of non-vertical minimal cylinders in $\Nil$ constructed from matrices $C_0$
with full entries}
\begin{Example}\label{ex:cyl3}
 We consider a periodic matrix 
 $C_0$ with full entries$:$
\begin{equation}\label{eq:diagC03}
C_0(z) =
  \begin{pmatrix} 
  \cosh (\sin z) & \sinh (\sin z) \\ 
  \sinh (\sin z) & \cosh (\sin z)
   \end{pmatrix}.
\end{equation}
 The minimum period $p$ of $C_0(z)$ is obviously $p = 2 \pi$.
 Then we compute 
\begin{align} 
 \zeta_0 (z)& =  C_0 (z)^{-1} \partial_z C_0(z)dz = 
 \begin{pmatrix}
  0 & \cos z \\ 
 \cos z & 0
  \end{pmatrix} dz,
\end{align}
 thus $\kappa (z) = \cos z$. It is also easy to see that the 
 signed are of $m$ in \eqref{eq:ell1}
 vanishes since $a_0$ and $b_0$ are real on $\R$. 
 Let us choose the complex function $h$ defined by 
\begin{equation}\label{eq:choiceh3}
 h(z) = \frac12 \cos z + (\cos z) \operatorname{sech} (2 \sin z) - i \sin (3 z), 
\end{equation}
 and define an $\isu$-potential $\zeta$ as
 \begin{equation}\label{eq:zetapotential3}
 \zeta (z, \lambda) = 
  \begin{pmatrix} 
  0 &   \lambda^{-1} h(z) + \lambda(\kappa (z)- h(z)) \\
   \lambda^{-1}(\overline{\kappa (\bar z)}- \overline{h(\bar z)}) + \lambda \overline{h(\bar z)} & 0 
   \end{pmatrix} dz.
 \end{equation}
  Since the first closing condition is always satisfied by our construction we only need to consider the second and  the third closing condition. 
 For this we compute the function $\alpha$ and the signed area of $m$ in \eqref{eq:alpha} and \eqref{eq:m}. It is easy to see that the signed area of $m$ vanishes, since $m$ is real on $\R$.
 On the one hand, the function $\alpha$ in \eqref{eq:alpha} can be computed as 
\begin{equation}\label{eq:alpha3}
 \alpha(t) =  2(\cos t - i \sin 3 t),
\end{equation}
 which is the same function in Example \ref{ex:cyl1}.
 Thus the second and third closing conditions
 can be satisfied.
 
\begin{figure}[htbp]
   \begin{center}
    \begin{tabular}{c}

      \begin{minipage}{0.5\hsize}
        \begin{center}
   \includegraphics[width=1.\textwidth]{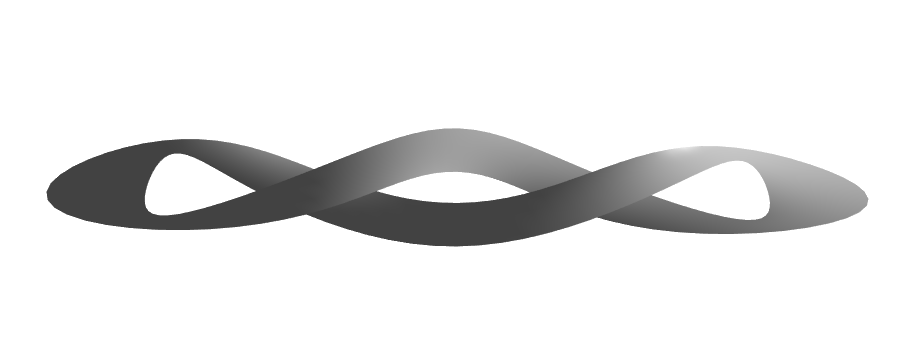}
        \end{center}
      \end{minipage}

      \begin{minipage}{0.5\hsize}
        \begin{center}
   \includegraphics[width=1.\textwidth]{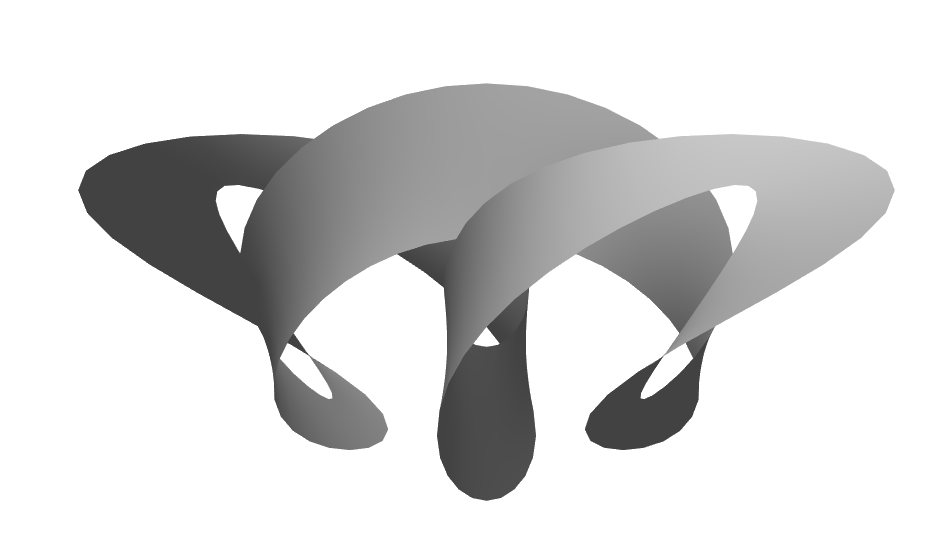}
        \end{center}
      \end{minipage}

    \end{tabular}
    \caption{Two views of the same non-vertical minimal cylinders in $\Nil$
 by the $\isu$-potential in \eqref{eq:zetapotential3}
 with $h$ in \eqref{eq:choiceh3} and identity initial condition. 
 The figures are constructed by \cite{Br:Matlab}.}
    \label{fig:3}
  \end{center}
\end{figure}

\end{Example}
Finally we will give an example of non-vertical minimal cylinders in $\Nil$ constructed from the non-diagonal matrix $C_0$.
\begin{Example}\label{ex:cyl4}
 We consider a periodic matrix 
 $C_0$ with full entries$:$
\begin{equation}\label{eq:diagC04}
C_0(z) =
  \begin{pmatrix} 
  e^{-i z}\cosh (\sin z) & \sinh (\sin z) \\ 
  \sinh (\sin z) & e^{i z}\cosh (\sin z)
   \end{pmatrix}.
\end{equation}
 The minimum period $p$ of $C_0(z)$ is obviously $p = 2 \pi$.
 Then we compute 
\begin{align*} 
 \zeta_0 (z)& =  C_0 (z)^{-1} \partial_z C_0(z)dz \\
  &= 
 \begin{pmatrix}
  -i \left\{\cosh (\sin z)\right\}^2 & e^{i z} 
  \left\{\cos z - \frac{i}2\sinh (2\sin z ) \right\} \\ 
 e^{-i z} \left\{\cos z + \frac{i}2\sinh (2\sin z) \right\}  & i \left\{\cosh (\sin z)\right\}^2
  \end{pmatrix} dz,
\end{align*}
 thus 
 \[
 \nu(z) =  -i \left\{\cosh (\sin z)\right\}^2, \quad \kappa (z) = e^{i z} 
  \left\{\cos z - \frac{i}2\sinh (2\sin z ) \right\}.
  \]
  It is also easy to see that the signed area $T$ of $m(t)$ 
  in \eqref{eq:m} can be computed as
  \[
   T = -\frac18\pi(I_{0}(4)-1) \fallingdotseq -4.04556,
  \]
   where $I_{\nu}(x)$ is the modified Bessel function of the first kind, 
 see Chapter III \cite{Bowman} for definition. Let us choose the complex function $h$ defined by 
\begin{equation}\label{eq:choiceh4}
    h(z)=-\frac{i}4 \left\{2 c_1 + 2 i \cos z+ \sinh(2 \sin z)\right\}, 
    \quad c_1 = \sqrt{|T|/\pi}
\end{equation}
 and define an $\isu$-potential $\zeta$ as
 \begin{equation}\label{eq:zetapotential4}
 \zeta (z, \lambda) = 
  \begin{pmatrix} 
  \nu(z) &   \lambda^{-1} h(z) + \lambda(\kappa (z)- h(z)) \\
   \lambda^{-1}(\overline{\kappa (\bar z)}- \overline{h(\bar z)}) + \lambda \overline{h(\bar z)} & - \nu (z) 
   \end{pmatrix} dz.
 \end{equation}
  Since the first closing condition is always satisfied by our construction we only need to consider the second and  the third closing condition. 
 For this we compute the function $\alpha$ \eqref{eq:alpha}.
  As we have shown that the signed are of $m$ is zero. On the one hand, 
 the function $\alpha$ in \eqref{eq:alpha} can be computed as 
\begin{equation}\label{eq:alpha4}
 \alpha(t) =  -i c_1 e^{-i t},
\end{equation}
 which is the derivative of the round circle about origin
 $\ell(t) = c_1 e^{-it}$ with radius $c_1$.
 Then the signed area of $\ell$ and $m$ are the same,
  and thus the second and third conditions can be satisfied.
\begin{figure}[htbp]
   \begin{center}
    \begin{tabular}{c}

      \begin{minipage}{0.5\hsize}
        \begin{center}
   \includegraphics[width=0.4
   \textwidth]{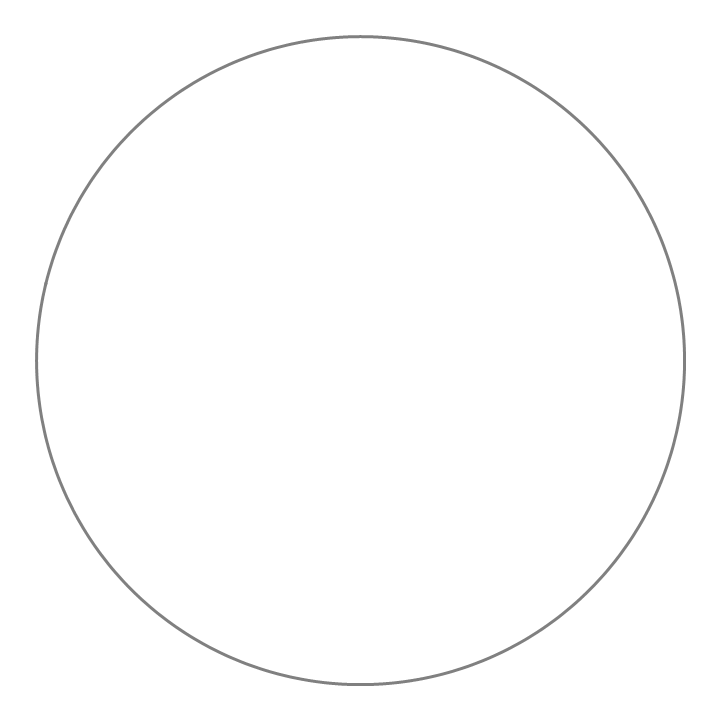}
        \end{center}
      \end{minipage}
      \begin{minipage}{0.5\hsize}
        \begin{center}
   \includegraphics[width=0.8\textwidth]{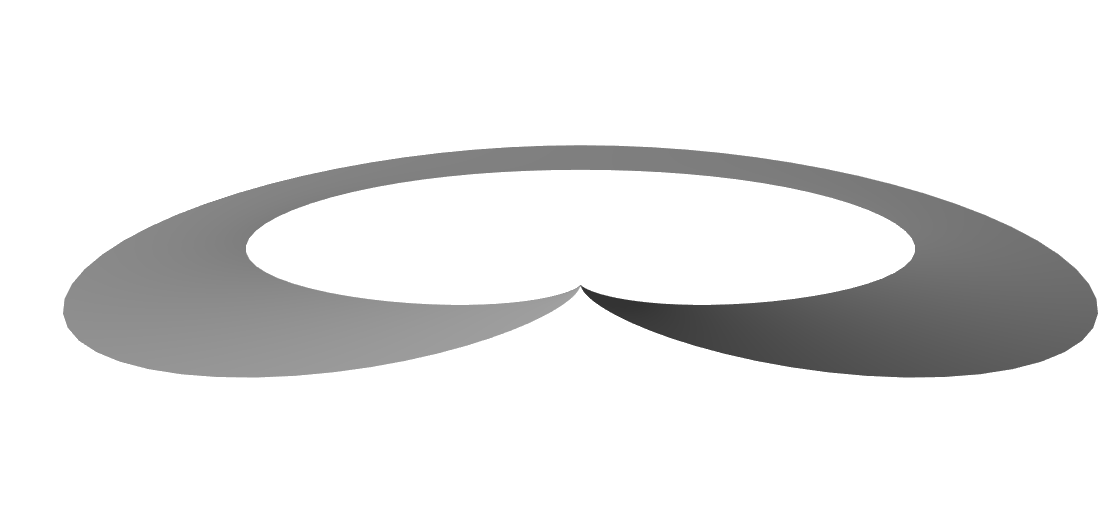}
        \end{center}
      \end{minipage}
    \end{tabular}
\includegraphics[width=0.3\textwidth]{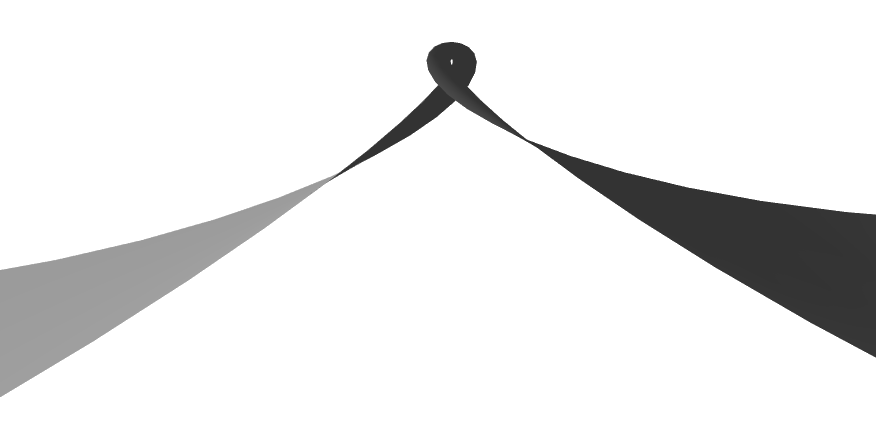}
    \caption{
    The plane curve $\ell$ (left, the round circle) and 
     the corresponding non-vertical minimal cylinder in $\Nil$ (right)
 by the $\isu$-potential in \eqref{eq:zetapotential4}
 with $h$ in \eqref{eq:choiceh4} and identity initial condition. 
 The bottom figure shows a close-up view of the right one.
 The figures are constructed by \cite{Br:Matlab}.
 }
    \label{fig:4}
  \end{center}
\end{figure}
\end{Example}

\begin{Remark}
 The Abresch-Rosenberg differential of these examples are 
 $B(z) = h(z) \overline{h(\bar z)} dz^2$,  and $h
(z)$ is given in \eqref{eq:choiceh0and1}, \eqref{eq:choiceh2} or \eqref{eq:choiceh3}.
\end{Remark}

\section{Spacelike CMC cylinders in $\Lt$}
The loop group technique used in \cite{DIKAsian},
\cite{DIK;Nil3-sym} and in this paper for the construction of minimal surfaces   in $\Nil$ uses the close relationship with spacelike CMC surfaces in Minkowski three-space $\Lt$ as manifestly expressed by the closely related, but different, Sym formulas.

Spacelike CMC surfaces in $\Lt$ have been investigated in some detail in \cite{BRS:Min}. We trust that the present reader will be able to translate  without problems notation and results from there into our setting. We therefore do not include any special notation nor any special results about this surface class into this note.

 In Theorem \ref{thm:framePeriodic}, we have shown that 
 the $\isu$-potentials considered there  produce  frame periodic minimal immersions in $\Nil$ and vice versa.
 It is obvious that the periodic matrix $C_0$
 stated in \eqref{eq:frameperiodic1}
 also produces 
 a spacelike CMC-immersion in $\Lt$ by the formula defining $f_{\Lt}$ in \eqref{eq:SymMin} inserting the
  extended frame $F$,
 and thus we have the following.
 \begin{Theorem}\label{thm:framePeriodicMin}
 Let $C_0$ be given by \eqref{eq:frameperiodic1}.
 Set $\zeta_0 = C_0^{-1} d C_0$ and define 
 a $\isu$-potential $\zeta$ by 
 \eqref{eq:frameperiodicPot}, where $h$ is an arbitrary holomorphic function along $\R \cap \D$.
 Then $\zeta$ defines a frame 
 periodic spacelike CMC-immersion into $\Lt$.
 Moreover, every frame periodic CMC-immersion in $\Lt$ 
 can be obtained this way.
 \end{Theorem}
 It is known (see e.g. \cite{BRS:Min}) that the closing conditions for $f_{\Lt}$ at $\lambda =1$
 can be phrased as follows:
\begin{equation}\label{eq:Minclosings}
 M(\l=1) = \pm \id \quad\mbox{and}\quad
 (\partial_{\l}M)(\l=1) = 0.
\end{equation}
 Since $X(\l)$ is defined by  $- i \l (\partial_{\l} M) M^{-1}$, the above 
 closing conditions can be rephrased as
\begin{equation}\label{eq:Minclosings2}
M(\l=1) = \pm \id \quad\mbox{and}\quad
 X(\l=1)=0.
\end{equation}
 Since for the frame periodic spacelike CMC-immersions
 we already have $M(\l) \in \LISU$ and 
 $M(\l=1)=\pm \id$, and  
 thus only the second closing condition  $X(\l=1)=0$ remains.
 Therefore we have the following corollary.
\begin{Corollary}
 Let $\zeta$ be an $\isu$-potential as stated in \eqref{eq:frameperiodicPot}.
 Define  functions $\alpha$ and $\beta$ by the equations
\begin{align}\label{eq:alphaMin}
 \alpha(t) &= a_0(t)^2 (2h(t) -\kappa(t)) + b_0 (t)^2(2\overline{h(t)} -\overline{k(t)}), \\
\label{eq:betaMin}
\beta (t) &= \Im \left\{a_0(t)\overline{b_0(t)} (2h(t) -\kappa(t))\right\},
\end{align}
 where $a_0$, $b_0$, $h$ and $\kappa$ are the functions occurring in \eqref{eq:frameperiodic1} 
 and  \eqref{eq:kappa}, respectively.
 Then $\zeta$  defines a spacelike CMC-cylinder in $\Lt$ if and only if
\begin{equation}\label{eq:closingMin}
  \int_0^p\alpha(t) dt = \int_0^p\beta(t) dt =0 
\end{equation}
 hold.
\end{Corollary}
\begin{Remark}
 Note that the function $\alpha$ stated in \eqref{eq:alphaMin} above is the same function as defined in \ref{eq:alpha}. 
\end{Remark}
 It is easy  to see that for the potentials $\zeta$ occurring in 
 Examples \ref{ex:cyl1} and \ref{ex:cyl2}, the closing conditions stated in 
 \eqref{eq:closingMin} are satisfied.
 
 Let us rephrase the closing conditions
  in \eqref{eq:closingMin}  more geometrically.
 First let us consider the plane curve $\ell$ in \eqref{eq:ell1} as before: 
 then $\ell(t) = \int_0^t \alpha(s) ds$ and 
\[
 2 h(z)- \kappa (z) = \frac{1 }{a_0 (z) \overline{a_0(\bar z)} + b_0 (z) \overline{b_0(\bar z)}}\left(
  \overline{a_0(\bar z)}^2 \alpha(z) - b_0(z)^2 \overline{\alpha(\bar z)}\right),
\]
 hold.
 We now simplify the condition $\int_0^p \beta (t) dt =0$. After plugging $2 h - \kappa$  into 
 $\beta$ a straightforward computation  using 
 the relation $a_0 (z) \overline{a_0(\bar z)} - b_0 (z) \overline{b_0(\bar z)}=1$,
 yields 
\[
\beta (z) = \Im \left(\overline{m(\bar z)} \alpha(z) \right),
\]
 where $m(z) = a_0(z) b_0(z)$.
  Therefore $\int_0^p \beta (t) dt =0$ is equivalent to that 
 \[
 \Im \langle {\ell}^\prime, m \rangle = 0,
 \]
 i.e., the imaginary part of the $L^2$-inner product of $m$ with  $\alpha$ vanishes. Similar to Theorem \ref{thm:main}, we have the following.
\begin{Theorem}\label{thm:mainspacelike}
 Let $\zeta$ be an $\isu$-potential as stated in \eqref{eq:frameperiodicPot}, 
 and let $\ell$ and $m$ be the curves in $\C$ as defined by  \eqref{eq:ell1}
 and $m=a_0 b_0$ with the entries of $C_0$ in \eqref{eq:frameperiodic1}, respectively. 
 Then $\zeta$, via the generalized Weierstrass representation, 
 gives a spacelike CMC-cylinder in $\Lt$
 if $\ell$ is closed and 
 the imaginary part of the $L^2$-inner product of $\ell^{\prime}$ and 
 $m$ vanishes.

 Conversely for closed analytic plane curves $\tilde \ell$ and $\tilde m$
  such that the imaginary part of the $L^2$-inner product of $\tilde \ell^{\prime}$ and $\tilde m$ vanishes, there exist 
   a periodic matrix $C_0$ and an $\isu$-potential $\zeta$ 
   such that the curves $\ell$ in 
   \eqref{eq:ell1} and $m=a_0 b_0$ are $\tilde \ell$ and $\tilde m$, 
    respectively  and the generalized Weierstrass
   representation yields a spacelike CMC-cylinder in $\Lt.$ 
 Moreover, all spacelike CMC-cylinders in $\Lt$ can be constructed in this way. 
\end{Theorem}

\begin{Example}[Two spacelike CMC cylinders -Two minimal cylinders in $\Nil$ revisited]
We will illustrate our results for spacelike CMC surfaces 
by reusing the data we had chosen in the $\Nil$ case, namely two potentials for a diagonal periodic matrix $C_0$, with diagonal entries $e^{icz}$ and $e^{-icz}$, and discuss the cases $c=1$ and $c = 0$.

 Let $\zeta$ be an $\isu$-potential as stated in \eqref{eq:zetapotential} with $C_0$ as just above. 
  As we have already seen that in Section 
  \ref{subsc:diagonal}, from Corollary \ref{closingcondition},
   $X(\lambda=1)$ is  off-diagonal, and thus 
   the second condition 
   $\int_0^p \beta(t) dt = 0$  in \eqref{eq:closingMin} holds anyway. 
 As before we choose a periodic holomorphic function $h$ as
\begin{align}
\label{eq:cmch1} h(z) &= \cos(z)-i \sin (3 z)\quad \mbox{for} \quad c=0, \\
\label{eq:cmch2} h(z) & = \exp(-i \pi/4)+\sqrt{6} \cos (4z) \quad \mbox{for} \quad c=1. 
 \end{align}
 Then the first condition $\int_0^p \alpha(t) dt = 0$ 
  in \eqref{eq:closingMin} also holds since we have 
  shown them in Examples \ref{ex:cyl1} and \ref{ex:cyl2}.
\begin{figure}[htbp]
  \begin{center}
    \begin{tabular}{c}

      \begin{minipage}{0.5\hsize}
        \begin{center}
   \includegraphics[width=0.7\textwidth]{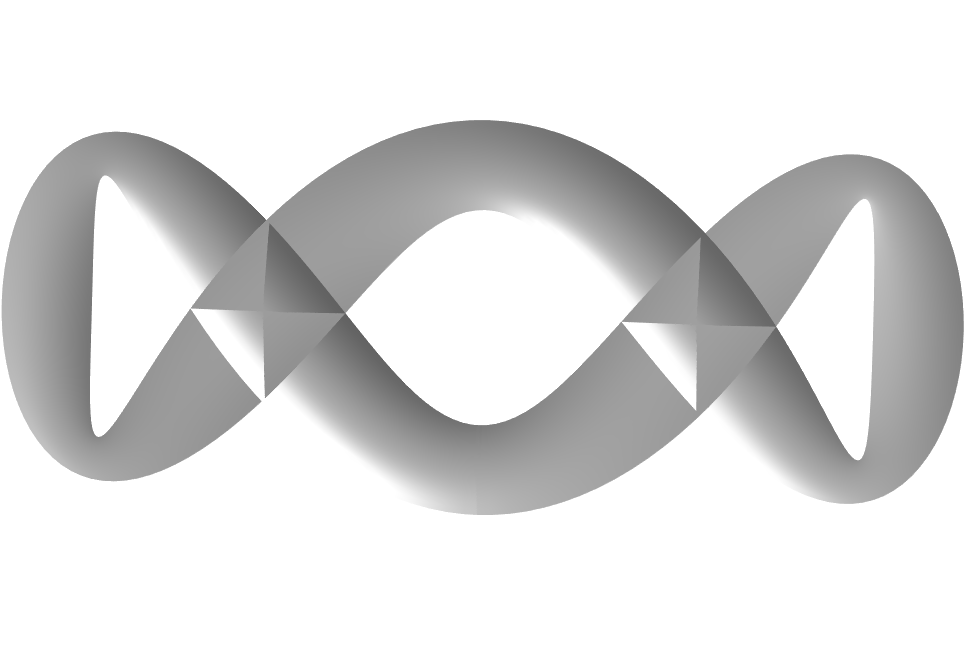}
        \end{center}
      \end{minipage}

      \begin{minipage}{0.5\hsize}
        \begin{center}
   \includegraphics[width=1\textwidth]{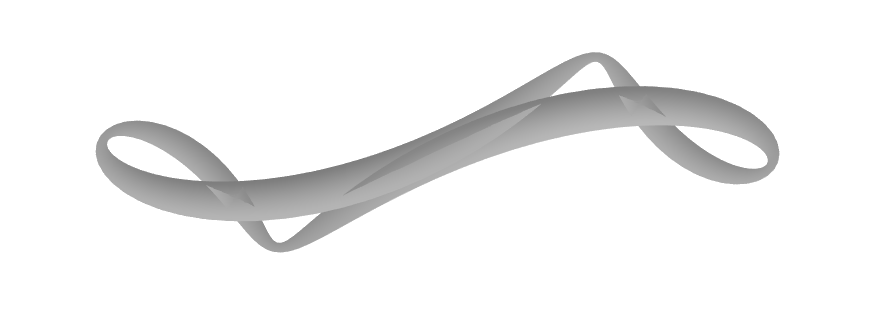}
        \end{center}
      \end{minipage}

    \end{tabular}
    \caption{A spacelike CMC-cylinders in $\Lt$ from 
  the $\isu$-potentials $\zeta$ in \eqref{eq:zetapotential}
 with $h$ in \eqref{eq:cmch1} (the left figure) or
 \eqref{eq:cmch1} (the right figure) with identity initial condition.
 The figure is constructed by \cite{Br:Matlab}.
}
    \label{fig:sCMC}
  \end{center}
\end{figure}
\end{Example}

\appendix
\section{Basic results}
In this appendix we will collect first basic definitions and will then present some results enabling us to use holomorphic potentials in place of meromorphic potentials.
 \subsection{Notation and definitions}
 We first define the twisted $\SL$ loop group as a space of continuous maps from $\mathbb{S}^1$ to the Lie group $\SL$, that is, 
 $\LSL =\{g : \mathbb{S}^1 \to \SL \;|\; g(-\lambda) = \sigma g(\lambda) \}$, where $\sigma =\ad (\sigma_3)$.  We restrict 
 our attention to loops in $\LSL$ such that 
 the associate Fourier series of the loops are absolutely convergent.
 Such loops determine a Banach algebra, the so-called  
 \textit{Wiener algebra}, and 
 it induces a topology on $\LSL$, 
 the so-called  \textit{Wiener topology}.
 From now on, we consider only $\LSL$ equipped with the Wiener topology.
 
 Let $D^{\pm}$ denote respective the inside of unit disk and
 the union of outside of the unit disk and infinity.
 We define \textit{plus} and \textit{minus} loop subgroups of $\LSL$;
$\LSLPM=\{ g \in \LSL \;|\; \mbox{$g$ can be extended holomorphically to $D^{\pm}$} \}$.
 By $\LSLPI$ we denote the subgroup of 
 elements of $\LSLP$ which take the value identity at zero.
 Similarly, by $\LSLMI$ we denote the subgroup of 
 elements of $\LSLM$ which take the value identity at infinity.
 
 We also define the $\ISU$-loop group as follows:  
 \[\LISU =\left\{ g  \in \LSL  \;|\;  \sigma_3 \overline{g(1/\bar \lambda)}^{t-1} \sigma_3 
 = g(\lambda)\right\}.\]
  It will be convenient to use $\tau (g)(\lambda) =   \sigma_3 \overline{g(1/\bar \lambda)}^{t-1} \sigma_3 $
  for $\lambda \in S^1$.
   For all geometric quantities, we can assume $\lambda \in \C^*$.
  
  For potentials $\zeta = \hat \zeta dz$ with  Lie algebra elements $\hat \zeta \in \lsl $ we say that $\zeta$ is \textit{an $\isu$-potential}  if $\tau(\hat \zeta)(\lambda) = -  \hat \zeta(\lambda)$ holds.


\subsection{Holomorphic $C(z,\lambda)$}
We should probably be a bit more careful about where from we choose our variables:
at one hand it is convenient to choose $\D$ as complex plane or unit disk, since they are invariant under complex conjugation. But since we consider already very early periodic functions with real period, we actually are interested in strips $\St$.
There are three types of strips: the whole complex plane, the upper half-plane and a strip of finite width.
In the first and the last case we can assume w.l.g. that the strip contains the real line \textquote{in the middle}. Meaning the real line for the case of $\St = \C$, and the real line for a strip of type
$-c_0 < \Im (z) < c_0$. 
For the third case, the upper half-plane $\mathbb{H} =\{ z  \mid \Im(z) > 0\}$, it is perhaps best
to let $\mathbb{Z}$ translate parallel to the real axis and to choose the base point to be $i$, instead of moving the upper-half plane down so that the new domain covers the actual real line.

\begin{Proposition}
 Let $f : \mathcal{C} \rightarrow \Nil$ be a minimal (immersed) cylinder in $\Nil$
 with universal cover $\St$.
 Then there exists a maximal minimal (immersed) cylinder in $\Nil$ prolonging $\St$.
\end{Proposition}
For the notion of \textquote{prolongable} see \cite[p.207]{Kra}.


\begin{proof}
Clearly, if $\St = \C$, no extension is possible and the given cylinder already is maximal.
If $\St$ is a finite strip with $\R$ in the middle, we consider all extensions, realized by finite strips with $\R$ in the middle. Then by Zorn's Lemma there exists a maximal object.
If $\St$ is the upper half plane, then any extension has a universal cover which can be realized as a \textquote{shifted down} upper half plane. Now a procedure as in the last case yields a maximal element.
\end{proof}

\subsection{Holomorphic $\isu$-potentials}
 For this we recall that the extended coordinate frame 
 has a meromorphic extension.
 Let's have a closer look:
 The right upper corner of the Maurer-Cartan form of the coordinate 
 frame $F$ is of the form
 $-\lambda^{-1}e^{u/2}  -\lambda \bar B e^{-u/2}$,
 where $u$ is the metric exponent and $B$ the Hopf differential 
 (see \cite{DIKAsian}).
 Since $F$ has a meromorphic extension to $\D \times \D$, also the expression above has such an extension. Since $F$ is real-analytic, also $e^{u/2}$ is real analytic. Therefore there exists an open subset ${\mathfrak D}_0$ of $\D \times \D$ containing $\D^\sharp = \{(z, \bar z); z \in \D \}$ to 
which $F$ extends holomorphically. We thus obtain:

\begin{Proposition}
 There exists an open subset $\mathfrak D_0$ of $\D \times \D$ containing $\D^\sharp = \{(z, \bar z); z \in \D \}$ to which $F$ extends holomorphically in $(z,w)$. 
 In particular, also $\zeta$ extends holomorphically in $(z,w)$ to $\mathfrak D_0$.
\end{Proposition}
\begin{Remark}
\mbox{}
\begin{enumerate}
\item If one starts from some potential like $\zeta$, then the maximal minimally immersed  
cylinder in $\Nil$ is defined on the largest strip on which the $\ISU$-Iwasawa decomposition is real analytic.

\item  Singularities along the boundary of such a  maximal strip indicate several different geometric 
features of the minimal cylinder.
\end{enumerate}
\end{Remark}

\bibliographystyle{plain}
\def\cprime{$'$}

\end{document}